\documentclass[10pt]{article}

\usepackage[utf8]{inputenc}
\usepackage[T1]{fontenc}

\usepackage{epsf}
\usepackage{amsmath}

\allowdisplaybreaks

\usepackage[showframe=false]{geometry}
\usepackage{changepage}

\usepackage{epsfig}
\usepackage{amssymb}

\usepackage{amsthm}
\usepackage{setspace}
\usepackage{cite}
\usepackage{mcite}

\usepackage{algorithmic}  
\usepackage{algorithm}

\usepackage{shadow}
\usepackage{fancybox}
\usepackage{fancyhdr}

\usepackage{color}
\usepackage[usenames,dvipsnames,svgnames,table]{xcolor}
\newcommand{\bl}[1]{\textcolor{blue}{#1}}

\definecolor{mypurple}{rgb}{.4,.0,.5}

\usepackage[hyphens]{url}

\usepackage[colorlinks=true,
            linkcolor=black,
            urlcolor=blue,
            citecolor=purple]{hyperref}

\usepackage{breakurl}

\def\y{{\bf y}}

\def\x{{\bf x}}

\def\x{{\mathbf x}}

\def\u{{\bf u}}

\def\x{{\bf x}}
\def\y{{\bf y}}

\def\q{{\bf q}}
\def\m{{\bf m}}
\def\a{{\bf a}}
\def\b{{\bf b}}
\def\c{{\bf c}}

\def\h{{\bf h}}

\def\be{\begin{equation}}
\def\ee{\end{equation}}
\def\ba{\left[\begin{array}}
\def\ea{\end{array}\right]}

\def\u{{\bf u}}

\def\x{{\bf x}}
\def\y{{\bf y}}

\def\q{{\bf q}}
\def\a{{\bf a}}
\def\b{{\bf b}}
\def\c{{\bf c}}

\def\p{{\bf p}}

\def\1{{\bf 1}}

\def\g{{\bf g}}
\def\0{{\bf 0}}

\def\calX{{\cal X}}
\def\calY{{\cal Y}}

\def\mR{{\mathbb R}}
\def\mN{{\mathbb N}}
\def\mE{{\mathbb E}}
\def\mS{{\mathbb S}}

\def\lp{\left (}
\def\rp{\right )}

\def\y{{\bf y}}

\def\x{{\bf x}}

\def\x{{\mathbf x}}

\def\u{{\bf u}}

\def\x{{\bf x}}
\def\y{{\bf y}}

\def\q{{\bf q}}
\def\a{{\bf a}}
\def\b{{\bf b}}
\def\c{{\bf c}}

\def\h{{\bf h}}

\def\be{\begin{equation}}
\def\ee{\end{equation}}
\def\ba{\left[\begin{array}}
\def\ea{\end{array}\right]}

\def\u{{\bf u}}

\def\x{{\bf x}}
\def\y{{\bf y}}

\def\q{{\bf q}}
\def\a{{\bf a}}
\def\b{{\bf b}}
\def\c{{\bf c}}

\def\p{{\bf p}}

\def\({\left (}
\def\){\right )}

\def\1{{\bf 1}}
\def\m{{\bf m}}
\def\q{{\bf q}}

\def\g{{\bf g}}
\def\0{{\bf 0}}

\def\cX{{\mathcal X}}
\def\cY{{\mathcal Y}}

\usepackage{xcolor}
\usepackage{color}

\definecolor{darkgreen}{rgb}{0, 0.4,0}

\definecolor{purplebrown}{rgb}{0.5,0.1,0.6}

\definecolor{ultclupcol}{rgb}{0.1,0.5,0.5}

\definecolor{mytrycolor}{rgb}{0.5,0.7,0.2}

\definecolor{ultclupcola}{rgb}{.5,0,.5}

\definecolor{shadebrown}{rgb}{0.1,0.1,0.9}
\definecolor{lightblue}{rgb}{0.2,0,1}

\usepackage{fancybox}
\usepackage{graphicx}
\usepackage{epstopdf}
\usepackage{epsfig}
\usepackage{wrapfig}
\usepackage{subfigure}

\usepackage{xcolor}
\usepackage{tcolorbox}
\tcbuselibrary{skins}

\newtcbox{\xmybox}{on line,
arc=7pt,
before upper={\rule[-3pt]{0pt}{10pt}},boxrule=0pt,
boxsep=0pt,left=6pt,right=6pt,top=0pt,bottom=0pt,enhanced, coltext=blue, colback=white!10!yellow}

\newtcbox{\xmyboxa}{on line,
arc=7pt,
before upper={\rule[-3pt]{0pt}{10pt}},boxrule=0pt,
boxsep=0pt,left=6pt,right=6pt,top=0pt,bottom=0pt,enhanced, colback=white!10!yellow}

\newtcbox{\xmyboxb}{on line,
arc=7pt,
before upper={\rule[-3pt]{0pt}{10pt}},boxrule=1pt,colframe=darkgreen!100!blue,
boxsep=0pt,left=6pt,right=6pt,top=0pt,bottom=0pt,enhanced, colback=white!10!yellow}

\newtcbox{\xmyboxc}{on line,
arc=7pt,
before upper={\rule[-3pt]{0pt}{10pt}},boxrule=.7pt,colframe=blue!100!blue,
boxsep=0pt,left=6pt,right=6pt,top=0pt,bottom=0pt,enhanced, coltext=blue, colback=white!10!yellow}

\newtcbox{\xmytboxa}{on line,
arc=7pt,
before upper={\rule[-3pt]{0pt}{10pt}},boxrule=.0pt,colframe=pink!50!yellow,
boxsep=0pt,left=6pt,right=6pt,top=0pt,bottom=0pt,enhanced, coltext=white, colback=blue!40!red}

\newtcbox{\xmytboxb}{on line,
arc=7pt,
before upper={\rule[-3pt]{0pt}{10pt}},boxrule=.0pt,colframe=pink!50!yellow,
boxsep=0pt,left=6pt,right=6pt,top=0pt,bottom=0pt,enhanced, coltext=white, colback=white!40!green}

\makeatletter
\newcommand\subsubsubsection{\@startsection{paragraph}{4}{\z@}{-2.5ex\@plus -1ex \@minus -.25ex}{1.25ex \@plus .25ex}{\normalfont\normalsize\bfseries}}
\newcommand\subsubsubsubsection{\@startsection{subparagraph}{5}{\z@}{-2.5ex\@plus -1ex \@minus -.25ex}{1.25ex \@plus .25ex}{\normalfont\normalsize\bfseries}}
\makeatother

\newtheorem{theorem}{Theorem}

\newtheorem{corollary}{Corollary}

\begin{document}

\begin{singlespace}

\title {Fully lifted random duality theory 
}
\author{
\textsc{Mihailo Stojnic
\footnote{e-mail: {\tt flatoyer@gmail.com}} }}
\date{}
\maketitle

\centerline{{\bf Abstract}} \vspace*{0.1in}

We study  a generic class of \emph{random optimization problems} (rops) and their typical behavior. The foundational aspects of the random duality theory (RDT), associated with rops, were discussed in \cite{StojnicRegRndDlt10}, where it was shown that one can often infer rops' behavior even without actually solving them. Moreover, \cite{StojnicRegRndDlt10}  uncovered that various quantities relevant to rops (including, for example, their typical objective values) can be determined (in a large dimensional context) even completely analytically. The key observation was that the \emph{strong deterministic duality} implies the, so-called, \emph{strong random duality} and therefore the full exactness of the analytical RDT characterizations. Here, we attack precisely those scenarios where the strong deterministic duality is not necessarily present and connect them to the recent progress made in studying bilinearly indexed (bli) random processes in \cite{Stojnicnflgscompyx23,Stojnicsflgscompyx23}. In particular, utilizing a fully lifted (fl) interpolating comparison mechanism introduced in  \cite{Stojnicnflgscompyx23}, we  establish corresponding \emph{fully lifted} RDT (fl RDT). We then rely on a stationarized fl interpolation realization introduced in \cite{Stojnicsflgscompyx23} to obtain complete \emph{statitionarized} fl RDT (sfl RDT).
A few well known problems are then discussed as illustrations of a wide range of practical applications implied by the generality of the considered rops.

\vspace*{0.25in} \noindent {\bf Index Terms: Random optimization/feasibility problems; Fully lifted random duality theory}.

\end{singlespace}

\section{Introduction}
\label{sec:back}

In this paper we study a generic class of random linearly constrained optimization problems. For a given function $f(\x):\mR^n\rightarrow\mR$ and a given set $\cX\subseteq \mR^n$, we consider
\begin{eqnarray}
\min_{\x\in\cX} & & f(\x)\nonumber \\
\mbox{subject to} & & A\x=0\nonumber \\
& & B\x\leq 0, \label{eq:lincons}
\end{eqnarray}
where $A\in\mR^{m_1\times n}$ is an $m_1\times n$ matrix  and $B\in\mR^{m_2\times n}$ is an $m_2\times n$ matrix. To facilitate the exposition, unless stated otherwise, we consider only the scenarios where the objective can be computed, i.e., where it exists and is bounded.

The program in (\ref{eq:lincons}) is a classical optimization problem that has been studied throughout rather vast literature over the last century (see, e.g., \cite{BV,Rock70,Luen84,Kuhn76,KuhnTuck51,BenNem01,BerTsi97,Bertsek99} and references therein). Various aspects of these types of problems are often of interest.
For example, within the standard  optimization theory treatments (e.g., \cite{BV,Rock70,Luen84,Kuhn76,KuhnTuck51,BenNem01,BerTsi97,Bertsek99}), when facing problem (\ref{eq:lincons}), most often, the primary goal is simply finding the minimal value of $f(\x)$ (and, if possible the $\x$'s that achieve it). The secondary goal is usually to do so in a computationally as efficient way as possible. If, say,  $f(\x)$ and $\cX$ are such that they allow for the solving process to be done in polynomial time, one usually, from an algorithmic point of view, considers  the above problem as solvable.

Our interest in this paper is slightly different from this classical approach and it relates to a somewhat different aspect of these problems. In particular, we don't necessarily look to find solutions of (\ref{eq:lincons}) or even to find its objective's optimal value. Instead,
we look at a special, random, class of problems from (\ref{eq:lincons}), and consider whether or not their analytical characterization is possible without actually solving them. For small $n$ (say $n=2$ or $n=3$), any given instance of (\ref{eq:lincons}) is typically solvable even fully analytically. We, however, consider large $n$ \emph{linear} regime, i.e. the regime where $\alpha_1=\lim_{n\rightarrow\infty} \frac{m_1}{n}$ and $\alpha_2=\lim_{n\rightarrow\infty} \frac{m_2}{n}$ and both $\alpha_1$ and $\alpha_2$ remain constants as $n$ grows. Solving (\ref{eq:lincons}) analytically in such a regime is in general very hard. Things, however, change if one moves to random mediums and views $A$ and $B$ as random matrices. These scenarios are precisely the ones initially considered in \cite{StojnicRegRndDlt10,StojnicCSetam09,StojnicUpper10,StojnicICASSP10block,StojnicICASSP10var} where the foundational principles of the so-called (\emph{non-lifted}) random duality theory (RDT) were introduced. Relying on these principles,  \cite{StojnicRegRndDlt10,StojnicCSetam09,StojnicUpper10,StojnicICASSP10block,StojnicICASSP10var} proved the existence of the so-called strong (\emph{non-lifted}) random duality for a subclass of problems given in (\ref{eq:lincons}) where the strong deterministic duality holds. This then turned out to be completely sufficient to enable a full analytical characterization of  the behavior of (\ref{eq:lincons}) even without actually solving it.

It soon became clear though that, despite their vast generality, the results  of \cite{StojnicRegRndDlt10,StojnicCSetam09,StojnicUpper10,StojnicICASSP10block,StojnicICASSP10var} might not ultimately be strong enough to handle more generic instances of (\ref{eq:lincons}). A long line of work then followed attempting to  improve on \cite{StojnicRegRndDlt10,StojnicCSetam09,StojnicUpper10,StojnicICASSP10block,StojnicICASSP10var}. In particular,
 \cite{StojnicLiftStrSec13,StojnicGardSphNeg13,StojnicGardSphErr13,StojnicAsymmLittBnds11,StojnicMoreSophHopBnds10} moved things to another level and introduced a \emph{partially lifted} upgrade, as a way of attacking the scenarios where the non-lifted RDT is not at its full power. We here proceed along the same lines and move things even further by establishing \emph{fully lifted} (fl) RDT. To do so,
we first recognize a connection between  the random linearly constrained optimization programs considered here and the recent progress made in studying bilinearly indexed (bli) random processes in \cite{Stojnicnflgscompyx23,Stojnicsflgscompyx23}. Recognition of such a connection allows us to ultimately introduce the fl RDT and its a complete \emph{stationarized} variant, sfl RDT.

In sections that follow, we slowly introduce all the needed ingredients to eventually establish the sfl RDT. Before proceeding with such a presentation, we mention, right here at the beginning, that our exposition does not rely on using any of the algorithmic or analytical results known for any specific (deterministic) instance of (\ref{eq:lincons}). This basically means that the reader is not really required to have any background in the optimization theory beyond a few classical concept that will be rather obvious from the presentation itself. Moreover, as it will become clear throughout the presentation, any such concepts will be fairly general and in no way tailored particularly for the classes of problems that we study here.

\section{Random linearly constrained programs}
\label{sec:randlincons}

As emphasized above, our focus is on random linearly constrained optimization program (\ref{eq:lincons}). To be more specific, we assume that all elements of matrices $A\in\mR^{m_1\times n}$ and $B\in\mR^{m_2\times n}$ are i.i.d. standard normals and for a given set $\cX\subseteq \mR^n$ and function $f(\x)$ consider
\begin{eqnarray}
\xi(f,\cX)=\min_{\x} & & f(\x)\nonumber \\
\mbox{subject to} & & A\x=0\nonumber \\
& & B\x\leq 0 \nonumber \\
& & \x\in\cX. \label{eq:randlincons1}
\end{eqnarray}
To facilitate writing, we assume throughout the paper that the objective of the above program exists and is either deterministically bounded (i.e., bounded for any particular realization of $A$ and $B$) or that it is bounded with overwhelming probability over the randomness of $A$ and $B$ (under overwhelming probability, we always assume a probability that goes to $1$ as $n\rightarrow \infty$). Moreover, to make the presentation less cumbersome, we also assume boundedness (again, either in deterministic or random sense) of all the functions of all the optimizing quantities that appear in the derivations below. With these assumptions in place, we proceed by transforming (\ref{eq:randlincons1})
\begin{eqnarray}
\xi(f,\cX) & = & \min_{\x\in\cX}\max_{\lambda\geq 0,\nu} \lp f(\x)+\nu^TA\x+\lambda^TB\x \rp
 = \min_{\x\in\cX}\max_{\lambda\geq 0,\nu} \lp f(\x)+\begin{bmatrix}\nu^T \lambda^T\end{bmatrix}\begin{bmatrix} A \\ B\end{bmatrix}\x \rp \nonumber \\
 & = & \min_{\x\in\cX}\max_{\y\in\cY} \lp f(\x)+\y^T G\x  \rp, \label{eq:randlincons2}
\end{eqnarray}
where $\nu\in\mR^{m_1},$ and $\lambda\in\mR^m_2,$ are column vectors, $G\in\mR^{m\times n},$ is a matrix comprised of i.i.d. standard normals, $m=m_1+m_2,$ and $\cY\subseteq \mR^m$ is comprised of  $\y\in\mR^m$ such that $\y_i\geq 0,m_1+1\leq i\leq m_1+m_2$. After recognizing that $\y^TG\x$ is a bli process, we are in position to utilize recent results \cite{Stojnicnflgscompyx23,Stojnicsflgscompyx23} (related to such processes) to characterize $\xi(f,\cX)$.

\subsection{Fully lifted interpolation}
\label{sec:lbrandlincons}

To be able to effectively use the results of \cite{Stojnicnflgscompyx23,Stojnicsflgscompyx23}, we first, for $r\in\mN$ and $k\in\{1,2,\dots,r+1\}$,  introduce function
\begin{equation}\label{eq:thm3eq1}
\psi(f,\calX,\calY,\p,\q,\m,\beta,s,t)  =  \mE_{G,{\mathcal U}_{r+1}} \frac{1}{\beta|s|\sqrt{n}\m_r} \log
\lp \mE_{{\mathcal U}_{r}} \lp \dots \lp \mE_{{\mathcal U}_2}\lp\lp\mE_{{\mathcal U}_1} \lp Z^{\m_1}\rp\rp^{\frac{\m_2}{\m_1}}\rp\rp^{\frac{\m_3}{\m_2}} \dots \rp^{\frac{\m_{r}}{\m_{r-1}}}\rp,
\end{equation}
where $\cX=\{\x^{(1)},\x^{(2)},\dots,\x^{(l)}\}$ and $\cY=\{\y^{(1)},\y^{(2)},\dots,\y^{(l)}\}$ for some $l\in\mN$ and where
\begin{eqnarray}\label{eq:thm3eq2}
Z & \triangleq & \sum_{i_1=1}^{l}\lp\sum_{i_2=1}^{l}e^{\beta D_0^{(i_1,i_2)}} \rp^{s} \nonumber \\
 D_0^{(i_1,i_2)} & \triangleq & f(\x^{(i_1)}) + \sqrt{t}(\y^{(i_2)})^T
 G\x^{(i_1)}+\sqrt{1-t}\|\x^{(i_1)}\|_2 (\y^{(i_2)})^T\lp\sum_{k=1}^{r+1}b_k\u^{(2,k)}\rp\nonumber \\
 & & +\sqrt{t}\|\x^{(i_1)}\|_2\|\y^{(i_2)}\|_2\lp\sum_{k=1}^{r+1}a_ku^{(4,k)}\rp +\sqrt{1-t}\|\y^{(i_2)}\|_2\lp\sum_{k=1}^{r+1}c_k\h^{(k)}\rp^T\x^{(i_1)},
\end{eqnarray}
and
\begin{eqnarray} \label{eq:thm3eq2a00}
 a_k & \triangleq \hspace{.1in}  a_k(\p,\q) &  \triangleq  \hspace{.1in}  \sqrt{\p_{k-1}\q_{k-1}-\p_k\q_k} \nonumber \\
 b_k & \triangleq \hspace{.1in}   b_k(\p,\q) &  \triangleq  \hspace{.1in} \sqrt{\p_{k-1}-\p_k} \nonumber \\
 c_k & \triangleq \hspace{.1in}  c_k(\p,\q) & \triangleq  \hspace{.1in}  \sqrt{\q_{k-1}-\q_k},
 \end{eqnarray}
with vectors $\m=[\m_0,\m_1,\m_2,...,\m_r,\m_{r+1}]$, $\p=[\p_0,\p_1,...,\p_r,\p_{r+1}]$, and $\q=[\q_0,\q_1,\q_2,\dots,\q_r,\q_{r+1}]$ being such that $\m_0=1$, $\m_{r+1}=0$,
 \begin{eqnarray}\label{eq:thm3eq2a0}
1\geq\p_0\geq \p_1\geq \p_2\geq \dots \geq \p_r\geq \p_{r+1} & = & 0 \nonumber \\
1\geq\q_0\geq \q_1\geq \q_2\geq \dots \geq \q_r\geq \q_{r+1} & = &  0,
 \end{eqnarray}
and ${\mathcal U}_k\triangleq [u^{(4,k)},\u^{(2,k)},\h^{(k)}]$ being such that the components of  $u^{(4,k)}\in\mR$, $\u^{(2,k)}\in\mR^m$, and $\h^{(k)}\in\mR^n$ are i.i.d. standard normals.

The above function $\psi(\cdot)$ is in fact directly connected to $\xi(f,\cX)$. For example, if say $a_k=0$, then
$\mE_{A,B}\xi(f,\cX)=\mE_{G}\xi(f,\cX)=-\sqrt{n}\lim_{\beta\rightarrow\infty} \psi(f,\calX,\calY,\p,\q,\m,\beta,-1,1)= -\sqrt{n}\psi(f,\calX,\calY,\p,\q,\m,\infty,-1,1)$. While, in the contexts of our interest in this paper, one has that generically $a_k\neq0$, it is still fairly clear that the connection is rather strong and that studying and understanding some of the properties of $\psi(\cdot)$ is likely to be very useful. The following fundamental result from \cite{Stojnicnflgscompyx23} is of key interest in that regard.
\begin{theorem}(\cite{Stojnicnflgscompyx23})
\label{thm:thm4}
 Let
  \begin{eqnarray}\label{eq:rthlev2genanal7a}
\zeta_r\triangleq \mE_{{\mathcal U}_{r}} \lp \dots \lp \mE_{{\mathcal U}_2}\lp\lp\mE_{{\mathcal U}_1} \lp Z^{\frac{\m_1}{\m_0}}\rp\rp^{\frac{\m_2}{\m_1}}\rp\rp^{\frac{\m_3}{\m_2}} \dots \rp^{\frac{\m_{r}}{\m_{r-1}}}, r\geq 1.
\end{eqnarray}
One then has
\begin{eqnarray}\label{eq:rthlev2genanal7b}
\zeta_k = \mE_{{\mathcal U}_{k}} \lp  \zeta_{k-1} \rp^{\frac{\m_{k}}{\m_{k-1}}}, k\geq 2,\quad \mbox{and} \quad
\zeta_1=\mE_{{\mathcal U}_1} \lp Z^{\frac{\m_1}{\m_0}}\rp,
\end{eqnarray}
with $\zeta_0=Z$ set for the completeness.  Moreover, consider the operators
\begin{eqnarray}\label{eq:thm3eq3}
 \Phi_{{\mathcal U}_k} & \triangleq &  \mE_{{\mathcal U}_{k}} \frac{\zeta_{k-1}^{\frac{\m_k}{\m_{k-1}}}}{\zeta_k},
 \end{eqnarray}
and set
\begin{eqnarray}\label{eq:thm3eq4}
  \gamma_0(i_1,i_2) & = &
\frac{(C^{(i_1)})^{s}}{Z}  \frac{A^{(i_1,i_2)}}{C^{(i_1)}} \nonumber \\
\gamma_{01}^{(r)}  & = & \prod_{k=r}^{1}\Phi_{{\mathcal U}_k} (\gamma_0(i_1,i_2)) \nonumber \\
\gamma_{02}^{(r)}  & = & \prod_{k=r}^{1}\Phi_{{\mathcal U}_k} (\gamma_0(i_1,i_2)\times \gamma_0(i_1,p_2)) \nonumber \\
\gamma_{k_1+1}^{(r)}  & = & \prod_{k=r}^{k_1+1}\Phi_{{\mathcal U}_k} \lp \prod_{k=k_1}^{1}\Phi_{{\mathcal U}_k}\gamma_0(i_1,i_2)\times \prod_{k=k_1}^{1} \Phi_{{\mathcal U}_k}\gamma_0(p_1,p_2) \rp.
 \end{eqnarray}
Let also
\begin{eqnarray}\label{eq:thm3eq5}
 \phi_{k_1}^{(r)} & = &
-s(\m_{k_1-1}-\m_{k_1}) \nonumber \\
&  & \times
\mE_{G,{\mathcal U}_{r+1}} \langle (\p_{k_1-1}\|\x^{(i_1)}\|_2\|\x^{(p_1)}\|_2 -(\x^{(p_1)})^T\x^{(i_1)})(\q_{k_1-1}\|\y^{(i_2)}\|_2\|\y^{(p_2)}\|_2 -(\y^{(p_2)})^T\y^{(i_2)})\rangle_{\gamma_{k_1}^{(r)}} \nonumber \\
 \phi_{01}^{(r)} & = & (1-\p_0)(1-\q_0)\mE_{G,{\mathcal U}_{r+1}}\langle \|\x^{(i_1)}\|_2^2\|\y^{(i_2)}\|_2^2\rangle_{\gamma_{01}^{(r)}} \nonumber\\
\phi_{02}^{(r)} & = & (s-1)(1-\p_0)\mE_{G,{\mathcal U}_{r+1}}\left\langle \|\x^{(i_1)}\|_2^2 \lp\q_0\|\y^{(i_2)}\|_2\|\y^{(p_2)}\|_2-(\y^{(p_2)})^T\y^{(i_2)}\rp\right\rangle_{\gamma_{02}^{(r)}}. \end{eqnarray}
Then
\begin{eqnarray}\label{eq:thm3eq6}
\frac{d\psi(f,\calX,\calY,\p,\q,\m,\beta,s,t)}{dt}  & = &       \frac{\mbox{sign}(s)\beta}{2\sqrt{n}} \lp  \lp\sum_{k_1=1}^{r+1} \phi_{k_1}^{(r)}\rp +\phi_{01}^{(r)}+\phi_{02}^{(r)}\rp.
 \end{eqnarray}
It particular, choosing $\p_0=\q_0=1$, one also has
\begin{eqnarray}\label{eq:rthlev2genanal43}
\frac{d\psi(f,\calX,\calY,\p,\q,\m,\beta,s,t)}{dt}  & = &       \frac{\mbox{sign}(s)\beta}{2\sqrt{n}} \sum_{k_1=1}^{r+1} \phi_{k_1}^{(r)} .
 \end{eqnarray}
 \end{theorem}
\begin{proof}
Follows immediately after trivial modifications of the proof presented in \cite{Stojnicnflgscompyx23}.
\end{proof}
Now, given that
\begin{eqnarray}\label{eq:fl2}
\psi(f,\calX,\calY,\p,\q,\m,\beta,s,1)=\psi(f,\calX,\calY,\p,\q,\m,\beta,s,0)+\int_{0}^{1}\frac{d\psi(f,\calX,\calY,\p,\q,\m,\beta,s,t)}{dt}dt,
\end{eqnarray}
the result of Theorem \ref{thm:thm4} extends the above mentioned connection between $\xi(f,\cX)$ and $\psi(f,\calX,\calY,\p,\q,\m,\infty,-1,1)$ so that now $\xi(f,\cX)$ is directly connected to  $\psi(f,\calX,\calY,\p,\q,\m,\infty,-1,0)$ as well. Still, two obstacles remain to be able to fully utilize this connection: i) handling the above mentioned $a_k\neq0$ scenario; and ii) the integration in (\ref{eq:fl2}). We discuss both of them in the next section.

\subsection{Stationarization}
\label{sec:stationarization}

Let $\psi_1$ be the following function
\begin{eqnarray}\label{eq:saip1}
\psi_1(f,\calX,\calY,\p,\q,\m,\beta,s,t) & = & -\frac{\mbox{sign}(s) s \beta}{2\sqrt{n}} \sum_{k=1}^{r+1}\Bigg(\Bigg. \p_{k-1}\q_{k-1}\mE_{G,{\mathcal U}_{r+1}} \langle\|\x^{(i_1)}\|_2\|\x^{(p_1)}\|_2\|\y^{(i_2)}\|_2\|\y^{(p_2)}\|_2\rangle_{\gamma_{k}^{(r)}}\nonumber \\
& & -\p_{k}\q_{k}\mE_{G,{\mathcal U}_{r+1}} \langle\|\x^{(i_1)}\|_2\|\x^{(p_1)}\|_2\|\y^{(i_2)}\|_2\|\y^{(p_2)}\|_2\rangle_{\gamma_{k+1}^{(r)}}\Bigg.\Bigg)
\m_{k} \nonumber \\
& & +\psi(f,\calX,\calY,\p,\q,\m,\beta,s,t),
 \end{eqnarray}
and consider the following system of equations
\begin{eqnarray}\label{eq:saip4}
\frac{d\psi_1(f,\calX,\calY,\p,\q,\m,\beta,s,t)}{d\p_{k_1}}
& = & 0, \nonumber \\
\frac{d\psi_1(f,\calX,\calY,\p,\q,\m,\beta,s,t)}{d\q_{k_1}}
& = & 0, \nonumber \\
\frac{d\psi_1(f,\calX,\calY,\p,\q,\m,\beta,s,t)}{d\m_{k_1}}
 & = & 0,
  \end{eqnarray}
where $k_1\in\{1,2,\dots,r\}$ and $\p_0(t)=\q_0(t)=\m_0(t)=1$. The following theorem is critically important.

\begin{theorem}(\cite{Stojnicsflgscompyx23})
\label{thm:thm5}
Assume completely stationirized  fully lifted random duality theory frame (\textbf{\emph{complete sfl RDT frame}}) of \cite{Stojnicsflgscompyx23} with $n\rightarrow\infty$ and $\bar{\p}(t)$, $\bar{\q}(t)$, and $\bar{\m}(t)$ as solutions of system (\ref{eq:saip4}). For $\bar{\p}_0(t)=\bar{\q}_0(t)=1$,
$\bar{\p}_{r+1}(t)=\bar{\q}_{r+1}(t)=\bar{\m}_{r+1}(t)=0$, and $\m_1(t)\rightarrow\m_0(t)=1$ one then has that
$\frac{d\psi_1(f,\calX,\calY,\bar{\p}(t),\bar{\q}(t),\bar{\m}(t),\beta,s,t)}{dt}   =   0$ and
\begin{eqnarray}\label{eq:thm5eq0}
 \lim_{n\rightarrow\infty}\psi_1(f,\calX,\calY,\bar{\p}(t),\bar{\q}(t),\bar{\m}(t),\beta,s,t)
& = &\lim_{n\rightarrow\infty}\psi_1(f,\calX,\calY,\bar{\p}(0),\bar{\q}(0),\bar{\m}(0),\beta,s,0) \nonumber \\
 & = & \lim_{n\rightarrow\infty}\psi_1(f,\calX,\calY,\bar{\p}(1),\bar{\q}(1),\bar{\m}(1),\beta,s,1).
 \end{eqnarray}
\end{theorem}
Finally, the operational version of the above theorem is the following practically relevant corollary.

\begin{corollary}(\cite{Stojnicsflgscompyx23})
\label{thm:thm6}
Assume the setup of Theorem \ref{thm:thm5}. Then
\begin{align}\label{eq:thm6eq1}
\lim_{n\rightarrow\infty}\psi_1(f,\calX,\calY,\bar{\p}(0),\bar{\q}(0),\bar{\m}(0),\beta,s,0)
& =  \lim_{n\rightarrow\infty}  \Bigg( \Bigg. -\frac{\mbox{sign}(s) s \beta}{2\sqrt{n}}
\sum_{k=1}^{r+1}\Bigg(\Bigg. \bar{\p}_{k-1}(0)\bar{\q}_{k-1}(0)
\nonumber \\
&  \times
\mE_{G,{\mathcal U}_{r+1}} \langle\|\x^{(i_1)}\|_2\|\x^{(p_1)}\|_2\|\y^{(i_2)}\|_2\|\y^{(p_2)}\|_2\rangle_{\gamma_{k}^{(r)}} \nonumber \\
&   -\bar{\p}_{k}(0)\bar{\q}_{k}(0)
 \mE_{G,{\mathcal U}_{r+1}} \langle\|\x^{(i_1)}\|_2\|\x^{(p_1)}\|_2\|\y^{(i_2)}\|_2\|\y^{(p_2)}\|_2\rangle_{\gamma_{k+1}^{(r)}}
   \Bigg.\Bigg)
\bar{\m}_k(0) \nonumber \\
&   + \psi(f,\calX,\calY,\bar{\p}(0),\bar{\q}(0),\bar{\m}(0),\beta,s,0) \Bigg.\Bigg),
\end{align}
with $\gamma$'s evaluated at $t=0$, and
\begin{align}\label{eq:thm6eq1a1}
\lim_{n\rightarrow\infty}\psi_1(f,\calX,\calY,\bar{\p}(1),\bar{\q}(1),\bar{\m}(1),\beta,s,1)
& =  \lim_{n\rightarrow\infty}  \Bigg( \Bigg. -\frac{\mbox{sign}(s) s \beta}{2\sqrt{n}}
\sum_{k=1}^{r+1}\Bigg(\Bigg. \bar{\p}_{k-1}(1)\bar{\q}_{k-1}(1)
\nonumber \\
&  \times
\mE_{G,{\mathcal U}_{r+1}} \langle\|\x^{(i_1)}\|_2\|\x^{(p_1)}\|_2\|\y^{(i_2)}\|_2\|\y^{(p_2)}\|_2\rangle_{\gamma_{k}^{(r)}} \nonumber \\
&   -\bar{\p}_{k}(1)\bar{\q}_{k}(1)
 \mE_{G,{\mathcal U}_{r+1}} \langle\|\x^{(i_1)}\|_2\|\x^{(p_1)}\|_2\|\y^{(i_2)}\|_2\|\y^{(p_2)}\|_2\rangle_{\gamma_{k+1}^{(r)}}
   \Bigg.\Bigg)
\bar{\m}_k(1) \nonumber \\
&   +  \psi(f,\calX,\calY,\bar{\p}(1),\bar{\q}(1),\bar{\m}(1),\beta,s,1) \Bigg.\Bigg),
  \end{align}
with $\gamma$'s evaluated at $t=1$. Moreover, let
 \begin{equation}\label{eq:thm6eq2}
\psi_S(f,\calX,\calY,\p,\q,\m,\beta,s,t)  =  \mE_{G,{\mathcal U}_{r+1}} \frac{1}{\beta|s|\sqrt{n}\m_r} \log
\lp \mE_{{\mathcal U}_{r}} \lp \dots \lp \mE_{{\mathcal U}_2}\lp\lp\mE_{{\mathcal U}_1} \lp Z_S^{\m_1}\rp\rp^{\frac{\m_2}{\m_1}}\rp\rp^{\frac{\m_3}{\m_2}} \dots \rp^{\frac{\m_{r}}{\m_{r-1}}}\rp,
\end{equation}
where,  analogously to (\ref{eq:thm3eq1}) and (\ref{eq:thm3eq2}),
\begin{equation}\label{eq:thm6eq3}
Z_S  \triangleq  \sum_{i_1=1}^{l}\lp\sum_{i_2=1}^{l}e^{\beta D_{0,S}^{(i_1,i_2)}} \rp^{s},
\end{equation}
and
\begin{eqnarray}\label{eq:thm6eq4}
 D_{0,S}^{(i_1,i_2)} &  \triangleq & f(\x^{(i_1)}) + \sqrt{t}(\y^{(i_2)})^T
 G\x^{(i_1)}+\sqrt{1-t}\|\x^{(i_1)}\|_2 (\y^{(i_2)})^T\lp\sum_{k=1}^{r+1}b_k\u^{(2,k)}\rp  \nonumber \\
  & &  +\sqrt{1-t}\|\y^{(i_2)}\|_2\lp\sum_{k=1}^{r+1}c_k\h^{(k)}\rp^T\x^{(i_1)}.
 \end{eqnarray}
Then
\begin{align}\label{eq:thm6eq5}
 \lim_{n\rightarrow\infty} \psi_S(f,\calX,\calY,\bar{\p}(1),\bar{\q}(1),\bar{\m}(1),\beta,s,1) & =
 \lim_{n\rightarrow\infty}\Bigg(\Bigg.
 -\frac{\mbox{sign}(s) s \beta}{2\sqrt{n}} \sum_{k=1}^{r+1}\Bigg(\Bigg. \bar{\p}_{k-1}(0)\bar{\q}_{k-1}(0)
 \nonumber \\
&  \times
\mE_{G,{\mathcal U}_{r+1}} \langle\|\x^{(i_1)}\|_2\|\x^{(p_1)}\|_2\|\y^{(i_2)}\|_2\|\y^{(p_2)}\|_2\rangle_{\gamma_{k}^{(r)}} \nonumber \\
&  -\bar{\p}_{k}(0)\bar{\q}_{k}(0)
\mE_{G,{\mathcal U}_{r+1}} \langle\|\x^{(i_1)}\|_2\|\x^{(p_1)}\|_2\|\y^{(i_2)}\|_2\|\y^{(p_2)}\|_2\rangle_{\gamma_{k+1}^{(r)}}  \Bigg.\Bigg)
\bar{\m}_k(0)
 \nonumber \\
&   + \psi_S(f,\calX,\calY,\bar{\p}(0),\bar{\q}(0),\bar{\m}(0),\beta,s,0) \Bigg.\Bigg), \nonumber \\
 \end{align}
with $\gamma$'s evaluated at $t=0$.
\end{corollary}
 \begin{proof}
   Follows through the same steps as Corollary 1 of \cite{Stojnicsflgscompyx23} by  disregarding its unit norm specializations and instead relying on general $\psi(\cdot)$ from (\ref{eq:saip1}).
 \end{proof}

\section{Practical utilization}
\label{sec:practical}

In this section we show how the above machinery can be practically utilized. To that end, we first observe
\begin{eqnarray}
\label{eq:fl3}
\lim_{n\rightarrow\infty} \psi_S(f,\calX,\calY,\bar{\p}(1),\bar{\q}(1),\bar{\m}(1),\infty,-1,1)
& = & \lim_{n\rightarrow\infty} \frac{\mE_G \max_{\x\in\cX} \lp - \max_{\y\in\cY}  \lp f(\x)+ \y^TG\x \rp\rp}{\sqrt{n}} \nonumber \\
& = & -\lim_{n\rightarrow\infty} \frac{\mE_G \min_{\x\in\cX} \max_{\y\in\cY}  \lp f(\x)+  \y^TG\x \rp}{\sqrt{n}} \nonumber \\
& = &
-\lim_{n\rightarrow\infty} \frac{\mE_G\xi(f,\cX)}{\sqrt{n}},
\end{eqnarray}
where the first equality follows from the definition of $\psi_S(f,\calX,\calY,\bar{\p}(1),\bar{\q}(1),\bar{\m}(1),\infty,-1,1)$ from (\ref{eq:thm6eq2})-(\ref{eq:thm6eq4}) and the last equality follows from (\ref{eq:randlincons2}). Under the boundedness assumption, one can, at $t=0$,  relying on both  the $\gamma$'s decoupling over $\x$ and $\y$ and the decoupling over the components of $\x$ and $\y$, discretize over possible values for $\|\x\|_2$ and $\|\y\|_2$. That means that all of the above discussion can be repeated by taking fixed (and ultimately optimal) values $\bar{x}=\|\x\|_2$ and $\bar{y}=\|\y\|_2$ which, in return, enables the following, more convenient form of (\ref{eq:thm6eq5}),
\begin{eqnarray}\label{eq:thm6eq6}
 \lim_{\beta,n\rightarrow\infty} \psi_S(f,\calX,\calY,\bar{\p}(1),\bar{\q}(1),\bar{\m}(1),\beta,s,1) & = &
 \lim_{\beta,n\rightarrow\infty} \Bigg(\Bigg.  -\frac{\mbox{sign}(s) s \beta}{2\sqrt{n}} \bar{x}^2\bar{y}^2   \nonumber \\
& &\times  \sum_{k=1}^{r+1}\Bigg(\Bigg. \bar{\p}_{k-1}(0)\bar{\q}_{k-1}(0)
   -\bar{\p}_{k}(0)\bar{\q}_{k}(0)
  \Bigg.\Bigg)
\bar{\m}_k(0)
 \nonumber \\
&  & + \psi_S(f,\cX_{\bar{x}},\cY_{\bar{y}},\bar{\p}(0),\bar{\q}(0),\bar{\m}(0),\beta,s,0) \Bigg.\Bigg),
 \end{eqnarray}
 where $\bar{\p}(0)$, $\bar{\q}(0)$, and $\bar{\m}(0)$ are evaluated from (\ref{eq:saip4}) for $\bar{x}=\|\x\|_2$ and $\bar{y}=\|\y\|_2$ and
 \begin{eqnarray}\label{eq:thm6eq6a1a0}
\cX_x & \triangleq & \{\x| \|\x\|_2=x,\x\in\cX\}  \nonumber \\
\cY_y & \triangleq & \{\y| \|\y\|_2=y,\y\in\cY\}.
\end{eqnarray}
Keeping in mind that $\bar{\m}_1(0)\rightarrow 1$
and
\begin{eqnarray}\label{eq:thm6eq6a1a1}
 \psi_S(f,\calX,\calY,\bar{\p}(0),\bar{\q}(0),\bar{\m}(0),\beta,s,0)=\psi(f,\calX,\calY,\bar{\p}(0),\bar{\q}(0),\bar{\m}(0),\beta,s,0),
  \end{eqnarray}
one then has that (\ref{eq:thm6eq1}) and (\ref{eq:thm6eq5}) imply that (\ref{eq:thm6eq6}) can  be rewritten for $s=-1$ as
\begin{eqnarray}\label{eq:thm6eq6a1a2}
 \lim_{n\rightarrow\infty} \psi_S(f,\calX,\calY,\bar{\p}(1),\bar{\q}(1),\bar{\m}(1),\infty,-1,1) & = &
 \lim_{\beta,n\rightarrow\infty} \Bigg(\Bigg.  \psi_1(f,\cX_{\bar{x}},\cY_{\bar{x}},\bar{\p},\bar{\q},\bar{\m},\beta,-1,0) \Bigg.\Bigg),
 \end{eqnarray}
where, $\bar{\p}_0=\bar{\q}_0=\bar{\m}_0=1$, $\bar{\m}_1\rightarrow 1$, and, similarly to (\ref{eq:saip4}), for $k_1\in\{1,2,\dots,r\}$, $\bar{\p},\bar{\q}$, and $\bar{\m}$ are the solutions to the following system of equations
\begin{eqnarray}\label{eq:thm6eq6a1a3}
\frac{d\psi_1(f,\cX_{\bar{x}},\cY_{\bar{y}},\p,\q,\m,\beta,-1,0)}{d\p_{k_1}}
& = & 0, \nonumber \\
\frac{d\psi_1(f,\cX_{\bar{x}},\cY_{\bar{y}},\p,\q,\m,\beta,-1,0)}{d\q_{k_1}}
& = & 0, \nonumber \\
\frac{d\psi_1(f,\cX_{\bar{x}},\cY_{\bar{y}},\p,\q,\m,\beta,-1,0)}{d\m_{k_1}}
 & = & 0.
  \end{eqnarray}
Since $\m_1\rightarrow 1$, we can  integrate out the most inner expectation of $\psi(\cdot)$ over $\x$ to obtain
\begin{eqnarray}\label{eq:thm6eq6a1a4}
\psi_1(f,\cX_{x},\cY_{y},\p,\q,\m,\beta,s,0)
& = &
  \Bigg(\Bigg.
 -\frac{\beta x^2y^2}{2\sqrt{n}} \lp 1   -\p_{1}\q_{1} \rp   -\frac{\beta  x^2y^2}{2\sqrt{n}}  \sum_{k=2}^{r+1}\Bigg(\Bigg. \p_{k-1}\q_{k-1}
   -\p_{k}\q_{k}
  \Bigg.\Bigg)
\m_k
 \nonumber \\
&  &
 +\frac{\beta x^2y^2}{2\sqrt{n}} \lp 1   -\q_{1}\rp
 +  \psi_{S,2}(f,\cX_x,\cY_x,\p,\q,\m,\beta,-1)    \Bigg.\Bigg), \nonumber \\
  \end{eqnarray}
where, similarly  to (\ref{eq:thm6eq2}),
 \begin{equation}\label{eq:thm6eq6a1a5}
\psi_{S,2}(f,\calX,\calY,\p,\q,\m,\beta,s)  =  \mE_{G,{\mathcal U}_{r+1}} \frac{1}{\beta|s|\sqrt{n}\m_r} \log
\lp \mE_{{\mathcal U}_{r}} \lp \dots \lp \mE_{{\mathcal U}_3}\lp\lp\mE_{{\mathcal U}_2} \lp Z_{S,2}^{\m_1}\rp\rp^{\frac{\m_2}{\m_1}}\rp\rp^{\frac{\m_3}{\m_2}} \dots \rp^{\frac{\m_{r}}{\m_{r-1}}}\rp,
\end{equation}
and,  analogously to (\ref{eq:thm3eq1}) and (\ref{eq:thm3eq2}) (or (\ref{eq:thm6eq3}) and (\ref{eq:thm6eq4})),
\begin{equation}\label{eq:thm6eq6a1a6}
Z_{S,2}  \triangleq  \sum_{x\in\cX}\lp\sum_{\y\in\cY}e^{\beta D_{0,S,2}^{(\x,\y)}} \rp^{s}, \nonumber \\
\end{equation}
but
\begin{eqnarray}\label{eq:thm6eq6a1a7}
 D_{0,S,2}^{(\x,\y)} &  \triangleq & f(\x) +\|\x\|_2 \y^T\lp\sum_{k=1}^{r+1}b_k\u^{(2,k)}\rp   +\|\y\|_2\lp\sum_{k=2}^{r+1}c_k\h^{(k)}\rp^T\x.
 \end{eqnarray}
A ground state regime rescaling, $\beta\rightarrow \bar{\beta}\sqrt{n}$, (where ($\bar{\beta}\rightarrow\infty$)) gives
$\bar{\beta}\bar{\m}_k=\bar{\c}_k$ ($\bar{\beta}\m_k=\c_k$) for $k>1$ and $\bar{\c}_1\rightarrow\bar{\m}_1\rightarrow 1$ (or $\c_1\rightarrow \m_1\rightarrow 1$) with $\bar{\c}_k$ (or $\c_k$) remaining constants as $\bar{\beta}$ grows. We then have
\begin{eqnarray}\label{eq:thm6eq6a1a8}
\lim_{\bar{\beta}\rightarrow\infty}\psi_1(f,\cX_{x},\cY_{y},\p,\q,\m,\beta,-1,0)
& = &
 \lim_{\bar{\beta}\rightarrow\infty}  \Bigg(\Bigg.
 -\frac{\bar{\beta} x^2y^2}{2} \lp 1   -\p_{1}\q_{1} \rp   -\frac{x^2y^2}{2}  \sum_{k=2}^{r+1}\Bigg(\Bigg. \p_{k-1}\q_{k-1}
   -\p_{k}\q_{k}
  \Bigg.\Bigg)
\c_k
 \nonumber \\
&  &
 +\frac{\bar{\beta} x^2y^2}{2} \lp 1   -\q_{1}\rp
 +  \psi_{S,2}(f,\cX_x,\cY_y,\p,\q,\m,\beta,-1)    \Bigg.\Bigg). \nonumber \\
  \end{eqnarray}
One now also observes
\begin{eqnarray}\label{eq:thm6eq6a1a9}
\lim_{\bar{\beta}\rightarrow\infty}
\frac{\psi_1(f,\cX_{x},\cY_{y},\p,\q,\m,\beta,-1,0)}{d\q_1}
& = &
 \lim_{\bar{\beta}\rightarrow\infty}  \Bigg(\Bigg.
 -\frac{\bar{\beta} x^2y^2}{2} \lp 1   -\p_{1}\rp  +O(1) \Bigg.\Bigg) =0, \nonumber \\
  \end{eqnarray}
which ensures
\begin{eqnarray}\label{eq:thm6eq6a1a9a0}
\bar{\p}_1=1-\frac{c_x}{\bar{\beta}}\rightarrow 1,
  \end{eqnarray}
where $c_x$ remains constant as $\bar{\beta}$ grows. One now also has
\begin{eqnarray}\label{eq:thm6eq6a1a10}
  \lim_{\bar{\beta}\rightarrow\infty}    \psi_{S,2}(f,\cX_x,\cY_y,\p,\q,\m,\beta,-1)
  & = &
  \lim_{\bar{\beta}\rightarrow\infty}    \psi_{S,3}(f,\cX_x,\cY_y,\p,\q,\m,\beta,-1)
  . \nonumber \\
  \end{eqnarray}
where
{\small  \begin{eqnarray}\label{eq:thm6eq6a1a11}
\psi_{S,3}(f,\calX,\calY,\p,\q,\m,\beta,s)  =  \mE_{G,{\mathcal U}_{r+1}} \frac{1}{\beta|s|\sqrt{n}\m_r} \log
\lp \mE_{{\mathcal U}_{r}} \lp \dots \lp \mE_{{\mathcal U}_3}\lp\lp\mE_{{\mathcal U}_2} \lp Z_{S,3}^{\m_1}\rp\rp^{\frac{\m_2}{\m_1}}\rp\rp^{\frac{\m_3}{\m_2}} \dots \rp^{\frac{\m_{r}}{\m_{r-1}}}\rp,
 \end{eqnarray}}

\noindent and,  analogously to (\ref{eq:thm3eq1}) and (\ref{eq:thm3eq2}),
\begin{eqnarray}\label{eq:fl5}
Z_{S,3} & \triangleq & e^{\beta D_{0,S,3}} \nonumber \\
 D_{0,S,3} & \triangleq  & \max_{\x\in\cX,\|\x\|_2=x} s \max_{\y\in\cY,\|\y\|_2=y}
 \lp f_{S}
+y    \lp\sum_{k=2}^{r+1}c_k\h^{(k)}\rp^T\x
+ x \y^T\lp\sum_{k=1}^{r+1}b_k\u^{(2,k)}\rp \rp, \nonumber
 \end{eqnarray}
 where $a_k\triangleq a_k(\p,\q)$, $b_k\triangleq b_k(\p,\q)$, and $c_k\triangleq c_k(\p,\q)$ are as in  (\ref{eq:thm3eq2a00}). Since
 $\bar{\p}_1=1-\frac{c_x}{\beta}\rightarrow 1$, one has $b_1=\sqrt{1-\bar{\p}_1}=\sqrt{\frac{c_x}{\beta}}\rightarrow 0$. Let $\bar{\y}$ be the solution of the above optimization. Then one can assume that $\bar{\y}$ depends only on those $k$ where $b_k=O(1)$. That also means that one can write
  \begin{eqnarray}\label{eq:thm6eq6a1a12}
\lim_{\beta\rightarrow\infty}\psi_{S,3}(f,\cX_x,\cY_y,\p,\q,\m,\beta,-1)
&  =  &\lim_{\beta\rightarrow\infty}
\lp\frac{\beta^2x^2b_1^2\|\bar{\y}\|_2^2}{2\beta\sqrt{n}}
+\psi_{S,\infty}(f,\cX_{x},\cY_{y},\p,\q,\m,x,y) \rp \nonumber \\
&  =  &\lim_{\beta\rightarrow\infty}
\lp \frac{\bar{\beta} x^2y^2(1-\p_1)}{2}
+\psi_{S,\infty}(f,\cX_{x},\cY_{y},\p,\q,\m,x,y) \rp, \nonumber \\
  \end{eqnarray}
 where
  \begin{eqnarray}\label{eq:fl4}
\psi_{S,\infty}(f_{S},\calX,\calY,\p,\q,\c,x,y)  =
 \mE_{G,{\mathcal U}_{r+1}} \frac{1}{n\c_r} \log
\lp \mE_{{\mathcal U}_{r}} \lp \dots \lp \mE_{{\mathcal U}_2}\lp\lp\mE_{{\mathcal U}_1} \lp Z_{S,\infty}^{\c_2}\rp\rp^{\frac{\c_3}{\c_2}}\rp\rp^{\frac{\c_4}{\c_3}} \dots \rp^{\frac{\c_{r}}{\c_{r-1}}}\rp \nonumber \\
 \end{eqnarray}
and,  analogously to (\ref{eq:thm3eq1}) and (\ref{eq:thm3eq2}),
\begin{eqnarray}\label{eq:fl5}
Z_{S,\infty} & \triangleq & e^{D_{0,S,\infty}} \nonumber \\
 D_{0,S,\infty} & \triangleq  & \max_{\x\in\cX,\|\x\|_2=x} - \max_{\y\in\cY,\|\y\|_2=y}
 \lp f_{S}
+y    \lp\sum_{k=2}^{r+1}c_k\h^{(k)}\rp^T\x^{(i_1)}
+ x (\y^{(i_2)})^T\lp\sum_{k=2}^{r+1}b_k\u^{(2,k)}\rp \rp \nonumber
\\& = & - \min_{\x\in\cX,\|\x\|_2=x}  \max_{\y\in\cY,\|\y\|_2=y}
 \lp f_{S}
+y    \lp\sum_{k=2}^{r+1}c_k\h^{(k)}\rp^T\x^{(i_1)}
+ x (\y^{(i_2)})^T\lp\sum_{k=2}^{r+1}b_k\u^{(2,k)}\rp \rp, \nonumber \\
 \end{eqnarray}
 where, as above and as in  (\ref{eq:thm3eq2a00}), $a_k\triangleq a_k(\p,\q)$, $b_k\triangleq b_k(\p,\q)$, and $c_k\triangleq c_k(\p,\q)$. Combining
(\ref{eq:thm6eq6a1a8}), (\ref{eq:thm6eq6a1a10}), and (\ref{eq:thm6eq6a1a12}), we find
\begin{align}\label{eq:thm6eq6a1a13}
\lim_{\bar{\beta}\rightarrow\infty}\psi_1(f,\cX_{x},\cY_{y},\p,\q,\m,\beta,-1,0)
& =
 \lim_{\bar{\beta}\rightarrow\infty}  \Bigg(\Bigg.
 -\frac{\bar{\beta} x^2y^2}{2} \lp 1   -\p_{1}\q_{1} \rp   -\frac{x^2y^2}{2}  \sum_{k=2}^{r+1}\Bigg(\Bigg. \p_{k-1}\q_{k-1}
   -\p_{k}\q_{k}
  \Bigg.\Bigg)
\c_k
 \nonumber \\
&  \quad
 +\frac{\bar{\beta} x^2y^2( 1   -\q_{1})}{2}
 + \frac{\bar{\beta} x^2y^2(1-\p_1)}{2}
+\psi_{S,\infty}(f,\cX_{x},\cY_{y},\p,\q,\m,x,y)   \Bigg.\Bigg). \nonumber \\
  \end{align}
One now also observes
\begin{eqnarray}\label{eq:thm6eq6a1a14}
\lim_{\bar{\beta}\rightarrow\infty}
\frac{\psi_1(f,\cX_{x},\cY_{y},\p,\q,\m,\beta,-1,0)}{d\p_1}
& = &
 \lim_{\bar{\beta}\rightarrow\infty}  \Bigg(\Bigg.
 -\frac{\bar{\beta} x^2y^2}{2} \lp 1   -\q_{1}\rp  +O(1) \Bigg.\Bigg) =0, \nonumber \\
  \end{eqnarray}
which ensures
\begin{eqnarray}\label{eq:thm6eq6a1a14a0}
\bar{\q}_1=1-\frac{c_y}{\bar{\beta}}\rightarrow 1,
  \end{eqnarray}
 where $c_y$ remains constant as $\bar{\beta}$ grows. Combining
(\ref{eq:thm6eq6a1a9a0}), (\ref{eq:thm6eq6a1a13}), and (\ref{eq:thm6eq6a1a14a0}), we find
\begin{align}\label{eq:thm6eq6a1a15}
\lim_{\bar{\beta}\rightarrow\infty}\psi_1(f,\cX_{x},\cY_{y},\p,\q,\m,\beta,-1,0)
& =
 \lim_{\bar{\beta}\rightarrow\infty}  \Bigg(\Bigg.
 -\frac{\bar{\beta} x^2y^2}{2} \lp 1   -\p_{1}\q_{1} \rp   -\frac{x^2y^2}{2}  \sum_{k=2}^{r+1}\Bigg(\Bigg. \p_{k-1}\q_{k-1}
   -\p_{k}\q_{k}
  \Bigg.\Bigg)
\c_k
 \nonumber \\
&  \quad
 +\frac{\bar{\beta} x^2y^2( 1   -\q_{1})}{2}
 + \frac{\bar{\beta} x^2y^2(1-\p_1)}{2}
+\psi_{S,\infty}(f,\cX_{x},\cY_{y},\p,\q,\m,x,y)   \Bigg.\Bigg) \nonumber \\
& =
    -\frac{x^2y^2}{2}  \sum_{k=2}^{r+1}\Bigg(\Bigg. \p_{k-1}\q_{k-1}
   -\p_{k}\q_{k}
  \Bigg.\Bigg)
\c_k
  +\psi_{S,\infty}(f,\cX_{x},\cY_{y},\p,\q,\m,x,y), \nonumber \\
  \end{align}
  and
\begin{equation}\label{eq:thm6eq6a1a16}
-\lim_{\bar{\beta}\rightarrow\infty}\psi_1(f,\cX_{x},\cY_{y},\p,\q,\m,\beta,-1,0)
  =
    \frac{x^2y^2}{2}  \sum_{k=2}^{r+1}\Bigg(\Bigg. \p_{k-1}\q_{k-1}
   -\p_{k}\q_{k}
  \Bigg.\Bigg)
\c_k
  -\psi_{S,\infty}(f,\cX_{x},\cY_{y},\p,\q,\m,x,y).
    \end{equation}
 After adding an outer optimization over $x$ and $y$, we obtain
\begin{eqnarray}\label{eq:thm6eq6a1a17}
-\lim_{\beta,n\rightarrow\infty}\psi_1(f,\cX_{\bar{x}},\cY_{\bar{y}},\bar{\p},\bar{\q},\bar{\m},\beta,-1,0)
&  = & \min_{x>0}\max_{y>0}
\lim_{n\rightarrow\infty} \Bigg(\Bigg.   \frac{x^2y^2}{2}  \sum_{k=2}^{r+1}\Bigg(\Bigg. \bar{\p}_{k-1}\bar{\q}_{k-1}
   -\bar{\p}_{k}\bar{\q}_{k}
  \Bigg.\Bigg)
\c_k \nonumber \\
& &   -\psi_{S,\infty}(f,\cX_{x},\cY_{y},\bar{\p},\bar{\q},\bar{\m},x,y) \Bigg.\Bigg),
    \end{eqnarray}
where $\bar{\p},\bar{\q}$, and $\bar{\m}$ are the solutions to the system of equations
(\ref{eq:thm6eq6a1a3}) evaluated for $\cX_x$ and $\cY_y$ with $x$ and $y$ being the running parameters of the above optimization. A combination of (\ref{eq:fl3}), (\ref{eq:thm6eq6a1a2}), and (\ref{eq:thm6eq6a1a17}) finally gives
\begin{eqnarray}\label{eq:thm6eq7}
\lim_{n\rightarrow\infty} \frac{\mE_G \xi(f,\cX)}{\sqrt{n}}
 & = &
\min_{x>0}\max_{y>0}
\lim_{n\rightarrow\infty} \lp  \frac{x^2y^2}{2}  \sum_{k=2}^{r+1}\Bigg(\Bigg. \p_{k-1}\q_{k-1}
   -\p_{k}\q_{k}
  \Bigg.\Bigg)
\c_k
  -\psi_{S,\infty}(f,\cX_{x},\cY_{y},\p,\q,\m,x,y) \rp. \nonumber \\
 \end{eqnarray}
The above results are summarized in the following theorem.
\begin{theorem}
\label{thm:thmsflrdt1}  Assume large $n$ regime and $\lim_{n\rightarrow\infty} \frac{m_1}{n}=\alpha_1$ and $\lim_{n\rightarrow\infty} \frac{m_2}{n}=\alpha_2$, where $\alpha_1$ and $\alpha_2$ are constants (i.e., independent of $n$). Let matrices $A\in\mR^{m_1\times n}$
and $B\in\mR^{m_2\times n}$ be comprised of i.i.d. standard normal elements. For a given function $f(\x):R^n\rightarrow R$, let $\xi(f,\cX)$ be the objective value of the optimization problem in (\ref{eq:randlincons1}). Assume complete sfl RDT frame from \cite{Stojnicsflgscompyx23} and that $f(\x)$ is such that $|\xi(f,\cX)|<\infty$ with overwhelming probability. Set
\begin{align}\label{eq:thmsflrdt2eq1}
   \psi_{rp} & \triangleq  \min_{\x\in\cX} \max_{\y\in\cY} \lp f(\x)+\y^TG\x \rp
   \qquad  \mbox{(\bl{\textbf{random primal}})} \nonumber \\
   \psi_{rd}(\p,\q,\c,x,y) & \triangleq    \frac{x^2y^2}{2}    \sum_{k=2}^{r+1}\Bigg(\Bigg.
   \p_{k-1}\q_{k-1}
   -\p_{k}\q_{k}
  \Bigg.\Bigg)
\c_k
  - \psi_{S,\infty}(f(\x),\calX,\calY,\p,\q,\c,x,y) \qquad \mbox{(\bl{\textbf{fl random dual}})}. \nonumber \\
 \end{align}
Let $\hat{\p}\triangleq \hat{\p}(x,y)$, $\hat{\q}\triangleq \hat{\q}(x,y)$, and  $\hat{\c}\triangleq \hat{\c}(x,y)$ be the solutions of the following system
\begin{eqnarray}\label{eq:thmsflrdt2eq2}
   \frac{d \psi_{rd}(\p,\q,\c,x,y)}{d\p} =  0,\quad
   \frac{d \psi_{rd}(\p,\q,\c,x,y)}{d\q} =  0,\quad
   \frac{d \psi_{rd}(\p,\q,\c,x,y)}{d\c} =  0.
 \end{eqnarray}
 Then,
\begin{eqnarray}\label{eq:thmsflrdt2eq3}
    \lim_{n\rightarrow\infty} \frac{\mE_G  \psi_{rp}}{\sqrt{n}}
  & = &
\min_{x>0} \max_{y>0} \lim_{n\rightarrow\infty} \psi_{rd}(\hat{\p}(x,y),\hat{\q}(x,y),\hat{\c}(x,y),x,y) \qquad \mbox{(\bl{\textbf{strong sfl random duality}})},\nonumber \\
 \end{eqnarray}
where $\psi_{S,\infty}(\cdot)$ is as in (\ref{eq:fl4}).
 \end{theorem}
\begin{proof}
Follows from previous discussion after recognizing that $\psi_1(\cdot)$ in (\ref{eq:saip4}) evaluated for $t=0$ and fixed $x$ and $y$, such that $\|\x\|_2=x$ and $\|\y\|_2=y$, is precisely $\psi_{rd}(\cdot)$.
\end{proof}

 Since the random primal concentrates, one can then write various corresponding probabilistic variants of (\ref{eq:thmsflrdt2eq3}). We, however, skip such exercises and, instead, note that the above theorem essentially connects $\psi_{rp}$ and $\psi_{rd}$. That basically means that if, say $\cX$ and $\cY$ are such that their elements are of fixed norms the above theorem holds for fixed $x$ and $y$, i.e., it holds without the optimization over $x$ and $y$ in (\ref{eq:thmsflrdt2eq3}). Various further generalizations and applications are possible as well. We below discuss a few interesting ones and defer a more complete discussion to separate papers that treat various specific problems in details.

\subsection{Non-homogeneous linear constraints}
\label{sec:nonhomcons}

The mechanism presented in the previous section is very powerful tool to characterize behavior of random linearly constrained optimization programs. It is also very generic and can be extended in many directions. As mentioned above, such extensions are typically problem specific and we defer studying them in details to separate papers. We do, however, mention here that all the tools that we use in those papers are conceptually already contained in what we presented above.

The example that we discuss below is a rather small modification of the above generic model and is provided to give a bit of a hint as to how relatively easily the whole framework can be massaged to fit various other scenarios. As it will be clear soon, all these modifications could have already been included in our original setup, albeit, at the expense of having a tone of tiny details overwhelm the presentation. In order to preserve the lightness of the exposition and to facilitate writing and reading, we opted to start with the simplest possible example and then build from there.

It is not that difficult to see that the set of linear constraints in (\ref{eq:lincons}) is scaling invariant (basically homogeneous). For any $\x$ that satisfies linear constraints of (\ref{eq:lincons}), $c\x$ does so as well as long as $c\geq 0$. Instead of homogeneous linear constraints one, usually, has the following, more general variant,
\begin{eqnarray}
\min_{\x} & & f(\x)\nonumber \\
\mbox{subject to} & & A\x=\a\nonumber \\
& & B\x\leq \b \nonumber \\
& & \x\in\cX, \label{eq:nonhomlincons1}
\end{eqnarray}
where $\a\in\mR^{m_1}$ and $\b\in\mR^{m_2}$. Vectors $\a$ and $\b$ can be random, deterministic (fixed), or a mixture of both. All possible scenarios can easily be handled by the above machinery. For the concreteness, we below sketch how the presented results adapt if $\a$ and $\b$ are deterministic. Almost the entire derivation presented in the previous section can be repeated line by line. In particular, one starts by defining the optimal value of the objective in (\ref{eq:nonhomlincons1}) as
\begin{eqnarray}
\xi_{nh}(f,\cX)=\min_{\x} & & f(\x)\nonumber \\
\mbox{subject to} & & A\x=\a\nonumber \\
& & B\x\leq \b \nonumber \\
& & \x\in\cX, \label{eq:nonhomlincons2}
\end{eqnarray}
and then writes the follwoing analogue to (\ref{eq:randlincons2})
\begin{eqnarray}
\xi_{nh}(f,\cX) & = & \min_{\x\in\cX}\max_{\lambda\geq 0,\nu} f(\x)+\nu^T A\x+\lambda^T B\x+\nu^T\a+\lambda^T\b \nonumber \\
& = & \min_{\x\in\cX}\max_{\y\in\cY} f(\x)+\y^TG\x+\y^T\g, \label{eq:nonhomlincons3}
\end{eqnarray}
where $\g=\begin{bmatrix}
          \a^T & \b^T
        \end{bmatrix}^T$.
One can then repeat the entire derivation from the previous section with very minimal and rather obvious modifications and obtain the following corollary of Theorem \ref{thm:thmsflrdt1}.

\begin{corollary}
\label{thm:thmsflrdt2}  Assume the setup of Theorem \ref{thm:thmsflrdt1} and sfl RDT frame of \cite{Stojnicsflgscompyx23}. Set
\begin{align}\label{eq:thmsflrdt3eq2}
   \psi_{rp}^{(nh)} & \triangleq  \min_{\x\in\cX} \max_{\y\in\cY} \lp f(\x)+\y^TG\x +\y^T\g \rp  \nonumber \\
   \psi_{rd}^{(nh)}(\p,\q,\c,x,y) & \triangleq    \frac{x^2y^2}{2}    \sum_{k=2}^{r+1}\Bigg(\Bigg.
   \p_{k-1}\q_{k-1}
   -\p_{k}\q_{k}
  \Bigg.\Bigg)
\c_k
  - \psi_{S,\infty}(f(\x)+\y^T\g,\calX,\calY,\p,\q,\c,x,y). \nonumber \\
 \end{align}
Let $\hat{\p}\triangleq \hat{\p}(x,y)$, $\hat{\q}\triangleq \hat{\q}(x,y)$, and  $\hat{\c}\triangleq \hat{\c}(x,y)$ be the solutions of the following system
\begin{eqnarray}\label{eq:thmsflrdt3eq2}
   \frac{d \psi_{rd}^{(nh)}(\p,\q,\c,x,y)}{d\p} =  0,\quad
   \frac{d \psi_{rd}^{(nh)}(\p,\q,\c,x,y)}{d\q} =  0,\quad
   \frac{d \psi_{rd}^{(nh)}(\p,\q,\c,x,y)}{d\c} =  0.
 \end{eqnarray}
 Then,
\begin{eqnarray}\label{eq:thmsflrdt3eq3}
    \lim_{n\rightarrow\infty} \frac{\mE_G  \psi_{rp}^{(nh)}}{\sqrt{n}}
  & = &
\min_{x>0} \max_{y>0} \lim_{n\rightarrow\infty} \psi_{rd}^{(nh)}(\hat{\p}(x,y),\hat{\q}(x,y),\hat{\c}(x,y),x,y),\nonumber \\
 \end{eqnarray}
where $\psi_{S,\infty}^{(nh)}(\cdot)$ is as in (\ref{eq:fl4}).
  \end{corollary}
\begin{proof}
Follows by repeating all the steps needed to show Theorem \ref{thm:thmsflrdt1}, with $f(\x)$ replaced by $f(\x)+\y^T\g$.
\end{proof}

\subsection{A couple of well known examples}
\label{sec:examples}

Besides further generalizations, various applications are possible as well. Along the same lines, we, below, show how the above discussion directly relates to several well known problems.

\subsubsection{Random feasibility problems}
\label{sec:rfps}

As recognized in \cite{StojnicGardGen13,StojnicGardSphErr13,StojnicGardSphNeg13,StojnicDiscPercp13}, the above machinery can be used to study not only random \emph{optimization} problems (rops), but also random \emph{feasibility} problems (rfps). If one, for example, takes $f(\x)=0$, then positivity and non-positivity of the objective value in (\ref{eq:nonhomlincons3}) ensure infeasibility and feasibility of (\ref{eq:nonhomlincons1}), respectively. For the concreteness, consider the following feasibility problem (basically the non-homogenous feasibility analogue to (\ref{eq:lincons}))
\begin{eqnarray}
\mbox{find} & & \x\nonumber \\
\mbox{subject to} & & A\x=\a\nonumber \\
& & B\x\leq \b \nonumber \\
& & \x\in\cX. \label{eq:ex1}
\end{eqnarray}
Clearly, randomness of any of $A$, $B$, $\a$, and $\b$ transforms (\ref{eq:ex1}) into a random feasibility problem (rfp) which can then be studied through the machineries of previous sections. First, one can start with (\ref{eq:nonhomlincons2}) (the nonhomogeneous variant of (\ref{eq:randlincons1})) with a generic $f(\x)$ and arrive to (\ref{eq:nonhomlincons3}) (the nonhomogeneous variant of (\ref{eq:randlincons2})). Specializing (\ref{eq:nonhomlincons3}) to the scenario where $f(\x)=0$ gives
\begin{eqnarray}
\xi_{feas}^{(0)}(f,\cX) = \min_{\x\in\cX} \max_{\y\in\cY}  \lp \y^T \begin{bmatrix}
                                                                A \\ B
                                                              \end{bmatrix}\x +\y^T\begin{bmatrix}
                                                                \a \\ \b
                                                              \end{bmatrix} \rp,
 \label{eq:ex2}
\end{eqnarray}
where we recall that set $\cY$ is a collection of $\y$ such that $\y_{i}\geq 0,m_1+1\leq i\leq m_1+m_2$. The observations made earlier in the corresponding homogenous case apply here as well. In particular, if there is an $\x$ such that $A\x=\a$ and $B\x\leq \b$, i.e., such that (\ref{eq:ex1}) is feasible, then the inner maximization in (\ref{eq:ex2})  makes $\xi_{feas}^{(0)}(f,\cX) =0$. On the other hand, nonexistence of such an $\x$ ensures that at least one of the equations in system $A\x=\a$ or at least one of the inequalities in $B\x\leq \x$ is not satisfied and the inner maximization can trivially make $\xi_{feas}^{(0)}(f,\cX) =\infty$. The insensitiveness with respect to $\y$ scaling observed in homogenous scenario also applies. That means that, from the feasibility point of view, $\xi_{feas}^{(0)}(f,\cX) =\infty$ and $\xi_{feas}^{(0)}(f,\cX) >0$ are equivalent, which allows restricting to $\|\y\|_2=1$ and ultimately ensures boundedness of $\xi_{feas}^{(0)}(f,\cX)$. Clearly, determining the objective of
\begin{eqnarray}
\xi_{feas}(f,\cX)
& =  &
\min_{\x\in\cX} \max_{\y\in\cY,\|\y\|_2=1}  \lp \y^T \begin{bmatrix}
                                                                A \\ B
                                                              \end{bmatrix}\x +\y^T\begin{bmatrix}
                                                                \a \\ \b
                                                              \end{bmatrix} \rp \nonumber \\
                                                              & =  &
\min_{\x\in\cX} \max_{\y\in\cY_1}  \lp \y^TG\x +\y^T\g \rp,
 \label{eq:ex3}
\end{eqnarray}
is then critically important in characterizing the rfps from (\ref{eq:ex1}). For the completeness, we, of course, recall that, as stated in (\ref{eq:thm6eq6a1a0}), $\cY_y \triangleq\{\y|\hspace{.03in} \|\y\|_2=y,\y\in\cY\}$, and consequently $\cY_1 \triangleq\{\y|\hspace{.03in} \|\y\|_2=1,\y\in\cY\}$.

\subsubsection{Perceptrons}
\label{sec:perc}

Handling random \emph{feasibility} problems (rfps) was also recognized in \cite{StojnicGardGen13,StojnicGardSphErr13,StojnicGardSphNeg13,StojnicDiscPercp13} as tightly connected to studying the so-called perceptrons -- the most fundamental components of any known neural network or machine learning systems. Various properties of perceptrons, either taken individually or as parts of larger network systems, are of interest and the above machinery can be utilized to thoroughly analyze all of them. Such a complete and detailed studying is problem specific though. As is by now a standard practice throughout the paper, in order to preserve the generic nature of the presented considerations, we refrain from the discussions related to particular examples. Instead, we defer specific, detail oriented, analyses to separate papers. Nonetheless, to provide a bit of concreteness and a small prelude to more complex discussions, we below make an exception and briefly highlight the so-called classifying/memorizing perceptrons' ability as an example of the the above machinery's utilization.

If, in (\ref{eq:ex3}), one chooses $\cX$ as any subset of $\mS^n$ (the $n$-dimensional unit sphere) and $m_1=0$ and $m_2=m$ (which implies $\cY_1=\mS_+^{m}$, where $\mS_+^m$ is the positive orthant portion of the $m$-dimensional unit sphere), then $\xi_{feas}(f,\cX)
=\xi_{perc}(f,\cX)$  can be viewed as an objective associated with the classical, so-called, generic perceptrons. Its role is of great  importance in establishing the so-called \emph{capacity} of the perceptrons when used as storage memories or classifiers. In particular, one can define such capacity as
\begin{eqnarray}
\alpha & = &  \lim_{n\rightarrow \infty} \frac{m_1+m_2}{n} = \lim_{n\rightarrow \infty} \frac{m}{n}  \nonumber \\
C_{perp} & \triangleq & \max \{\alpha |\hspace{.08in}  \xi_{feas}(f,\cX)=\xi_{feas}(f,\cX)>0\}.
  \label{eq:ex4}
\end{eqnarray}
Many versions of such perceptrons are of interest and have been studied throughout the literature over the last several decades. Below we select a few well known ones.

\subsubsection{Spherical perceptron}
\label{sec:sphperc}

Choosing additionally  $\cX=\mS^n$  in (\ref{eq:ex3}), one obtains the associated objective of the generic \emph{spherical} perceptrons. A further  particular choice, $\g=\b=-\kappa$ with $\kappa\geq 0$, produces the so-called positive (universal threshold) spherical perceptrons scenario (see, e.g.,  \cite{StojnicGardGen13,StojnicGardSphErr13,StojnicGardSphNeg13,GarDer88,Gar88,Schlafli,Cover65,Winder,Winder61,Wendel,Cameron60,Joseph60,BalVen87,Ven86,SchTir02,SchTir03,Talbook11a}), On the other hand, the choice $\kappa\leq 0$, produces the corresponding negative counterparts (see, e.g.,
\cite{StojnicGardSphErr13,StojnicGardSphNeg13,Talbook11a,Talbook11b,FraPar16,FraHwaUrb19,FraSclUrb19,ParUrbZam20,FPSUZ17,FraSclUrb20}). One then finds
\begin{eqnarray}
 C^{(sph)}_{perp} & = & \max \left \{\alpha |\hspace{.08in}  \xi^{(sph)}_{feas}\lp f,\mS^{n} \rp >0 \right \}.
  \label{eq:ex5}
\end{eqnarray}
In case of the positive spherical perceptrons,  $C^{(sph)}_{perp}$ was determined in
\cite{Cover65,Winder,Winder61,Wendel} for $\kappa=0$ through combinatorial geometric considerations and, decades later, through several different mathematical techniques and for any $\kappa\geq 0$ (see, \cite{StojnicGardGen13,StojnicGardSphErr13,SchTir02,SchTir03,Talbook11a}).  On the other hand, the negative case, turned out to be much harder. Talagrand in \cite{Talbook11a,Talbook11b} conjectured an upper bound which was proven as fully rigorous in  \cite{StojnicGardGen13}. Moreover, \cite{StojnicGardSphNeg13} lowered the upper bound showing that it is not tight but rather strict. In a separate paper, we will present in full details all the underlying technicalities and practical results that one obtains utilizing the above machinery and how they ultimately compare to those of  \cite{StojnicGardGen13,StojnicGardSphNeg13}. Also, various other aspects of spherical perceptrons (see, e.g., \cite{StojnicGardSphErr13}) are of interests and their full detailed treatment is deferred to separate papers as well.

\subsubsection{Binary perceptron}
\label{sec:binperc}

Choosing $\cX=\{-\frac{1}{\sqrt{n}},\frac{1}{\sqrt{n}}\}^n$  in (\ref{eq:ex3}) (i.e., choosing $\cX$ as the corners of the unit norm hypercube), one obtains the associated objective of the generic \emph{binary} perceptrons. A further  particularization, $\g=\b=-\kappa$ with $\kappa\geq 0$ or $\kappa<0$, produces their positive or negative (universal threshold) variants (see, e.g., \cite{StojnicGardGen13,GarDer88,Gar88,StojnicDiscPercp13,KraMez89,GutSte90,KimRoc98,DingSun19,BoltNakSunXu22,NakSun23}).  One then finds
\begin{eqnarray}
 C^{(bin)}_{perp} & = & \max \left \{\alpha |\hspace{.08in}  \xi^{(bin)}_{feas}\lp f,\left \{-\frac{1}{\sqrt{n}},\frac{1}{\sqrt{n}}\right \}^n\rp>0\right \}.
  \label{eq:ex5}
\end{eqnarray}
For $\kappa=0$, utilizing  statistical physics based replica theory, \cite{GarDer88,Gar88} gave some initial rough (replica symmetric) $ C^{(bin)}_{perp}$ estimates, while  \cite{KraMez89} established a more refined prediction $ C^{(bin)}_{perp}\approx0.833$. A substantial progress, towards establishing the \cite{KraMez89}'s replica predictions as fully rigorous lower bounds, was made in \cite{DingSun19}, and later on strengthened in \cite{NakSun23} (see, also, \cite{BoltNakSunXu22} for related small density results and \cite{CXu21} for the capacity's sharp threshold existence).

\subsubsection{Compressed sensing phase transitions}
\label{sec:l1sec}

In \cite{DonohoPol} and a bit later on in \cite{StojnicCSetam09,DonMalMon09,BayMon10} the so-called \emph{weak}  compressed sensing phase transitions (PT) were established. In particular, \cite{StojnicCSetam09} (and a bit earlier, \cite{StojnicICASSP09}) established that set $\cX$ of the following form plays a critical role
 \begin{equation}
\mS_{weak}^n=\left \{\x| \hspace{.05in} \sum_{i=1}^{n-k}|\x_i|\leq \sum_{i=n-k+1}^{n}\x_i, \x\in\mS^n\right \}.\label{eq:ex6}
\end{equation}
where $k\leq m_1$ and $\beta=\lim_{n\rightarrow \infty}\frac{k}{n}$ is a constant independent of $n$. Choosing $\g=0$, $m_1=m$, and $m_2=0$ in (\ref{eq:ex3}) and setting $\alpha=\lim_{n\rightarrow \infty}\frac{m}{n}=\lim_{n\rightarrow \infty}\frac{m_1}{n}$, \cite{StojnicCSetam09} determined, for any given $\beta$, the critical weak $\ell_1$ compressed sensing $\alpha$ phase transition
\begin{eqnarray}
 \alpha^{(PT,weak)} & \triangleq & \max \left \{\alpha |\hspace{.08in}  \xi_{feas}^{(weak)}(f,\mS_{weak}^n)>0 \right \}.
  \label{eq:ex7}
\end{eqnarray}
Moreover, \cite{DonohoPol} also established lower bounds on the substantially more challenging, so-called, strong and sectional $\ell_1$ phase transitions. \cite{StojnicCSetam09} established corresponding alternative ones and \cite{StojnicLiftStrSec13} lifted and improved over both \cite{StojnicCSetam09} and \cite{DonohoPol} bounds. For example, for sectional PTs, \cite{StojnicLiftStrSec13,StojnicCSetam09}
established that set $\cX$ of the following form
 \begin{equation}
\mS_{sec}^n=\left \{\x| \hspace{.05in} \sum_{i=1}^{n-k}|\x_i|\leq \sum_{i=n-k+1}^{n}|\x_i|, \x\in\mS^n\right \},\label{eq:ex8}
\end{equation}
plays a critical role in determining sectional $\ell_1$ compressed sensing $\alpha$ phase transition
\begin{eqnarray}
 \alpha^{(PT,sec)} & \triangleq & \max \left \{\alpha |\hspace{.08in}  \xi_{feas}^{(sec)}(f,\mS_{sec}^n)>0 \right \}.
  \label{eq:ex9}
\end{eqnarray}

The above examples are very few illustrative ones provided to give a hint as to how wide the range of applications of our results actually is. Further studying of the above presented machinery in the context of each of these applications is therefore of great interest and will be explored in separate papers.

\section{Summary of the fully lifted (fl) RDT formalism}
\label{sec:summary}

Since the introduced machinery is a very powerful mathematical engine, we find it useful to formalize its main steps in a neat way to make it more approachable and easier to use. As the discussion from Sections \ref{sec:randlincons} and \ref{sec:practical} revealed, there are four essential steps that are the foundational components of the entire formalism. The last three of them are clearly visible in Theorem \ref{thm:thmsflrdt1} and relate to the concepts of random primal/dual and strong sfl random duality. The first one basically incorporates algebraic characterizations of the underlying problems that (if possible) transform them into forms on which fl RDT is applicable. Clearly, this step is directly preceding all the other three.  A clear summary of all the four steps (precisely in the order how they should be done) that ensure proper application of all the fl RDT concepts is highlighted below.

\vspace{-.0in}\begin{center}
\tcbset{colback=orange!40!white!40!yellow,colframe=blue!75!black,fonttitle=\bfseries,title style={left color=black, right color=cyan},interior style={left color=yellow!10!white,right color=yellow!80!white}}
\begin{tcolorbox}[beamer,title=Fully lifted RDT -- summary of the main concepts \cite{Stojnicnflgscompyx23,Stojnicsflgscompyx23}, width=5.7in]
\vspace{-.1in}
 \emph{\begin{enumerate}
\item  \bl{\textbf{Algebraic characterization and creation of (fully lifted) random primal}}
  \item \bl{\textbf{Determining fully lifted random dual}}
 \item \bl{\textbf{Handling fully lifted random dual}}
\item \bl{\textbf{Ensuring strong fully lifted random duality.}}
  \end{enumerate}}
  \vspace{-.2in}$ $
\end{tcolorbox}
\end{center}\vspace{-.1in}

The above, of course, in a way, resembles the main principles of the plain (non-lifted) RDT itself from \cite{StojnicRegRndDlt10,StojnicCSetam09,StojnicUpper10,StojnicICASSP10block,StojnicICASSP10var} and the corresponding ones of the (partially) lifted RDT from, e.g., \cite{StojnicLiftStrSec13,StojnicGardSphNeg13,StojnicGardSphErr13,StojnicAsymmLittBnds11,StojnicMoreSophHopBnds10}. We below emphasize the connections and differences among all three.

\subsection{Connections between \emph{fully} lifted, \emph{partially} lifted,  and \emph{non}-lifted RDTs}
\label{sec:summaryconnect}

In scenarios where $f(\x)$ and $\cX$ are such that the strong deterministic duality holds, i.e. such that (\ref{eq:randlincons2}) can be further rewritten as
\begin{eqnarray}
\xi(f,\cX) & = & \min_{\x\in\cX}\max_{\lambda\geq 0,\nu} \lp f(\x)+\nu^TA\x+\lambda^TB\x \rp
 = \min_{\x\in\cX}\max_{\lambda\geq 0,\nu} \lp f(\x)+\begin{bmatrix}\nu^T \lambda^T\end{bmatrix}\begin{bmatrix} A \\ B\end{bmatrix}\x \rp \nonumber \\
 & = & \min_{\x\in\cX}\max_{\y\in\cY} \lp f(\x)+\y^T G\x  \rp, \nonumber \\
 & = & \max_{\y\in\cY} \min_{\x\in\cX}\lp f(\x)+\y^T G\x  \rp,
  \label{eq:conn1}
\end{eqnarray}
then \cite{StojnicRegRndDlt10,StojnicCSetam09,StojnicUpper10,StojnicICASSP10block,StojnicICASSP10var} established the plain (non-lifted) RDT that can be summarized in the following \emph{non-lifted} corollary of Theorem \ref{thm:thmsflrdt1}.
\begin{corollary}
\label{cor:conn1} (Strong deterministic duality $\implies$ strong (non-lifted) random duality \cite{StojnicRegRndDlt10,StojnicCSetam09,StojnicUpper10,StojnicICASSP10block,StojnicICASSP10var}) Assume the setup of Theorem \ref{thm:thmsflrdt1} and that $f(\x)$ and $\cX$ are such that the strong deterministic duality from (\ref{eq:conn1}) holds. Then for $r=2$, $\hat{\c}\rightarrow[1,1,0,0]$, $\hat{\p}\rightarrow[1,1,0,0]$, and $\hat{\q}\rightarrow[1,1,0,0]$
 \begin{eqnarray}\label{eq:corconn1eq1}
    \lim_{n\rightarrow\infty} \frac{\mE_G  \psi_{rp}}{\sqrt{n}}
  & = &
\min_{x>0} \max_{y>0} \lim_{n\rightarrow\infty} \psi_{rd}(\hat{\p},\hat{\q},\hat{\c},x,y) \qquad \mbox{(\bl{\textbf{strong non-lifted random duality}})},\nonumber \\
 \end{eqnarray}
where $\psi_{S,\infty}(\cdot)$ is as in (\ref{eq:fl4}).
 \end{corollary}
\begin{proof}
Shown in \cite{StojnicRegRndDlt10,StojnicCSetam09,StojnicUpper10,StojnicICASSP10block,StojnicICASSP10var}, together with corresponding concentrating and probabilistic statements.
\end{proof}

In a less technical language, \cite{StojnicRegRndDlt10,StojnicCSetam09,StojnicUpper10,StojnicICASSP10block,StojnicICASSP10var} essentially established that the strong deterministic duality actually implies strong non-lifted random duality. In other words, sfl RDT simplifies to non-lifted RDT when the strong deterministic duality is in place. On the other hand, when the strong deterministic duality is not in place, then things need to be carefully reconsidered. A long line of work \cite{StojnicLiftStrSec13,StojnicGardSphNeg13,StojnicGardSphErr13,StojnicAsymmLittBnds11,StojnicMoreSophHopBnds10} made a substantial progress in that direction and established the following \emph{partially lifted} corollary of Theorem \ref{thm:thmsflrdt1}.
\begin{corollary}
\label{cor:conn2} (Partially lifted RDT \cite{StojnicLiftStrSec13,StojnicGardSphNeg13,StojnicGardSphErr13,StojnicAsymmLittBnds11,StojnicMoreSophHopBnds10}) Assume the setup of Theorem \ref{thm:thmsflrdt1}. Then for $r=2$, $\hat{\c}\rightarrow[1,1,\c_2,0]$, $\hat{\p}\rightarrow[1,1,0,0]$, and $\hat{\q}\rightarrow[1,1,0,0]$
 \begin{align}\label{eq:corconn2eq1}
   \psi_{rp} & \triangleq  \min_{\x\in\cX} \max_{\y\in\cY} \lp f(\x)+\y^TG\x \rp
   \qquad  \mbox{(\bl{\textbf{random primal}})} \nonumber \\
   \psi_{rd}(\hat{\p},\hat{\q},\c,x,y) & =   \frac{x^2y^2}{2}\c_2
  - \psi_{S,\infty}(f(\x),\calX,\calY,\hat{\p},\hat{\q},\c,x,y) \qquad \mbox{(\bl{\textbf{partially lifted random dual}})}. \nonumber \\
 \end{align}
Let  $\hat{\c}=\hat{\c}(x,y)$ be the solution of the following system
\begin{eqnarray}\label{eq:corconn2eq2}
    \frac{d \psi_{rd}(\hat{\p},\hat{\q},\c,x,y)}{d\c} =  0.
 \end{eqnarray}
 Then,
\begin{eqnarray}\label{eq:corconn2eq3}
    \lim_{n\rightarrow\infty} \frac{\mE_G  \psi_{rp}}{\sqrt{n}}
  & \geq  &
\min_{x>0} \max_{y>0} \lim_{n\rightarrow\infty} \psi_{rd}(\hat{\p},\hat{\q},\hat{\c}(x,y),x,y) \qquad \mbox{(\bl{\textbf{partially lifted random duality}})},\nonumber \\
 \end{eqnarray}
where $\psi_{S,\infty}(\cdot)$ is as in (\ref{eq:fl4}).
\end{corollary}
\begin{proof}
Shown in \cite{StojnicLiftStrSec13,StojnicGardSphNeg13,StojnicGardSphErr13,StojnicAsymmLittBnds11,StojnicMoreSophHopBnds10}, together with corresponding concentrating and probabilistic statements.
\end{proof}
Speaking a bit less technically, \cite{StojnicLiftStrSec13,StojnicGardSphNeg13,StojnicGardSphErr13,StojnicAsymmLittBnds11,StojnicMoreSophHopBnds10}  established that the partially lifted RDT produces valid lower bounds, which sometimes (say, when the strong deterministic duality holds) are tight. Finally, carefully looking at the above corollary (and, ultimately, the results of, e.g., \cite{StojnicLiftStrSec13,StojnicGardSphNeg13,StojnicGardSphErr13,StojnicAsymmLittBnds11,StojnicMoreSophHopBnds10}), one observes that the partial lifting of \cite{StojnicLiftStrSec13,StojnicGardSphNeg13,StojnicGardSphErr13,StojnicAsymmLittBnds11,StojnicMoreSophHopBnds10} corresponds to a partial variant of the second level of fully lifted RDT.

\section{Conclusion}
\label{sec:conc}

A class of famous linearly constrained optimization problems is considered and their behavior in a random medium is studied. These types of problems were initially considered in \cite{StojnicRegRndDlt10,StojnicCSetam09,StojnicUpper10,StojnicICASSP10block,StojnicICASSP10var} where foundational principles of (non-lifted) random duality theory (RDT) were established. In particular, for a subclass of problems where the strong deterministic duality holds, \cite{StojnicRegRndDlt10,StojnicCSetam09,StojnicUpper10,StojnicICASSP10block,StojnicICASSP10var} proved the existence of the so-called strong (non-lifted) random duality which then was sufficient to analytically fully characterize the behavior of the underlying problems, even without actually solving them.

A long line of work, \cite{StojnicLiftStrSec13,StojnicGardSphNeg13,StojnicGardSphErr13,StojnicAsymmLittBnds11,StojnicMoreSophHopBnds10}, went further and introduced a partially lifted upgrade as a way of attacking scenarios where the non-lifted RDT is not at its full power. Here, we further upgrade partially lifted RDT  to the \emph{fully lifted} (fl) one. In particular,
we first recognize a connection between  the random linearly constrained optimization programs of interest here and the recent progress made in studying bilinearly indexed random processes in \cite{Stojnicnflgscompyx23,Stojnicsflgscompyx23}. As a result of such a connection, we are then able to introduce the ultimate level of fully lifted (fl) RDT and its a complete stationarized variant, sfl RDT.

The whole machinery is presented in a generic way and is consequently a very powerful generic mathematical engine that can be applied in a plethora of different scenarios. To give a hint at the range of potential applications, we first showed that in addition to handling the random \emph{optimization} problems (rops), it can also be equally successfully used in handling random \emph{feasibility} problems (rfps). As concrete examples of applications, we selected a few well known feasibility problems related to modern topics in machine learning/neural networks and compressed sensing. These include the so-called  \emph{spherical} perceptrons   (see, e.g., \cite{StojnicGardGen13,StojnicGardSphErr13,StojnicGardSphNeg13,GarDer88,Gar88,Schlafli,Cover65,Winder,Winder61,Wendel,Cameron60,Joseph60,BalVen87,Ven86,SchTir02,SchTir03}), \emph{binary} perceptrons (see, e.g., \cite{StojnicGardGen13,GarDer88,Gar88,StojnicDiscPercp13,KraMez89,GutSte90,KimRoc98}), and the \emph{sectional} compressed sensing $\ell_1$ phase transitions (see, e.g., \cite{StojnicCSetam09,StojnicLiftStrSec13}).

The introduced methodology also allows various extensions and generalizations. For example, all those discussed in \cite{Stojnicnflgscompyx23,Stojnicsflgscompyx23}, including the ones related to the so-called Hopfield models (see, e.g., \cite{Hop82,PasFig78,Hebb49,PasShchTir94,ShchTir93,BarGenGueTan10,BarGenGueTan12,Tal98,StojnicMoreSophHopBnds10}), asymmetric Little models
(see, e.g.,  \cite{BruParRit92,Little74,BarGenGue11bip,CabMarPaoPar88,AmiGutSom85,StojnicAsymmLittBnds11})  and many others, are immediately applicable here. Moreover, many further generalizations and applications beyond those discussed in \cite{Stojnicnflgscompyx23,Stojnicsflgscompyx23} are possible as well. Studying such extensions, however, is often problem specific and requires a few technical modifications. We will present many interesting results obtained in these directions in separate papers.

Finally we should mention that the considered statistical context assumes the standard Gaussianity of the primal models. These assumptions make the presentation less cumbersome and clearer to follow. However, all presented results can be adapted to hold for way more general statistical ensembles. The central limit theorem and, in particular, its Lindeberg variant \cite{Lindeberg22}, can  be incorporated to show this. A particularly elegant approach in this regard can be found in \cite{Chatterjee06}.

\begin{singlespace}
\bibliographystyle{plain}
\bibliography{nflgscompyxRefs}
\end{singlespace}

\end{document}